\documentclass[11pt]{article}

\usepackage{amsmath,amssymb,amsthm,amsfonts}
\usepackage{mathtools}
\usepackage{mathrsfs}
\usepackage{bm}
\usepackage{graphicx}
\usepackage{float}
\usepackage{multirow}
\usepackage{booktabs}
\usepackage{longtable}
\usepackage{arydshln}
\usepackage{enumerate}
\usepackage{cite}
\usepackage{hyperref}
\usepackage[a4paper,margin=2.5cm]{geometry}

\hypersetup{
	colorlinks=true,
	linkcolor=blue,
	citecolor=blue,
	urlcolor=blue
}

\allowdisplaybreaks[2]

\numberwithin{equation}{section}

\newtheorem{theorem}{Theorem}[section]
\newtheorem{lemma}[theorem]{Lemma}
\newtheorem{proposition}[theorem]{Proposition}

\newtheorem{definition}[theorem]{Definition}

\newtheorem{remark}[theorem]{Remark}

\newcommand{\D}{\mathbb D}

\newcommand{\C}{\mathbb C}

\newcommand{\Hq}{\mathbb H} 
\newcommand{\R}{\mathbb R} 
\newcommand{\B}{\mathbb B}
\newcommand{\Sph}{\mathbb S}
\newcommand{\CI}{\mathbb C_I}
\newcommand{\op}{\operatorname{op}}
\newcommand{\diag}{\operatorname{diag}}
\newcommand{\RI}{\mathscr R_I}
\newcommand{\e}{\mathrm e}

\title{\Large\bf    Corona theorem for the quaternionic Hardy space
	\footnotetext{
		\endgraf Y. Lu was supported by the National Natural Science Foundation of China
		(Grant No. 12031002). C. Zu was supported by the National Natural Science
		Foundation of China (Grant No. 12401151), and the Postdoctoral Researcher
		Foundation of China (Grant No. GZB20240100).
}}

\author{
	Zhaopeng Lin,
	Yufeng Lu,
	Chao Zu\thanks{Corresponding author}
}

\date{}

\begin{document}

  \maketitle
  
  \vspace{-0.8cm}
  
  \begin{center}
  	\begin{minipage}{14cm}\small    {\noindent{\bf Abstract} \quad  
The finite-generator corona theorem for bounded slice regular functions on the quaternionic unit ball was recently established by Colombo, Pozzi, Sabadini, and Wick. In the present paper, we extend the quaternionic \(H^\infty\)-corona theorem to countably many generators and obtain quantitative estimates that are independent of the cardinality of the generating family. We also establish the corresponding \(H^p\)-corona theorem for the full range \(1\leq p<\infty\), with quantitative norm estimates for both finite and countable families of generators. In the Hilbert-space setting, we prove a quaternionic Leech factorization theorem for Hardy-space multipliers and derive, as a consequence, a   Toeplitz corona characterization. Our approach is based on a fixed-slice \(2\times2\) complex matrix realization of the slice regular product, together with operator-valued corona and factorization techniques.  
  			\endgraf 
  			{\bf Mathematics Subject Classification (2020).}\quad
Primary 30G35; Secondary 30H80, 47B35.
  			\endgraf
  			{\bf Keywords.}\quad
  		Corona theorem; slice regular
functions;  Quaternionic Hardy space.}
  	\end{minipage}
  \end{center}
   
\section{Introduction}
\label{sec:introduction}

The corona problem is a fundamental instance in which a topological
question concerning the maximal ideal space of a Banach algebra is
converted into an analytic division problem. Let
\(M(H^\infty(\D))\) denote the maximal ideal space of the Banach
algebra \(H^\infty(\D)\) of bounded holomorphic functions on the unit
disk. Every point \(z\in\D\) determines an evaluation character
\[
\varepsilon_z(f)=f(z),
\qquad f\in H^\infty(\D),
\]
and hence gives a natural embedding
\(\D\hookrightarrow M(H^\infty(\D))\). Kakutani asked whether the
image of \(\D\) is dense in \(M(H^\infty(\D))\). This density question is equivalent to a B\'ezout problem for bounded
analytic functions; see \cite{MR3329539} for a detailed historical
account.

More precisely, for finitely many functions
\(a_1,\ldots,a_N\in H^\infty(\D)\), the relevant analytic question is
whether a uniform pointwise lower bound
\[
0<\delta^2
\leq
\sum_{j=1}^N |a_j(z)|^2
\leq1,
\qquad z\in\D,
\]
forces the existence of
\(b_1,\ldots,b_N\in H^\infty(\D)\) satisfying
\[
\sum_{j=1}^N a_jb_j=1.
\]
Carleson answered this question affirmatively in his celebrated
corona theorem \cite{MR141789}. Beyond settling Kakutani's question,
Carleson's proof introduced geometric and measure-theoretic tools
that became central in complex and harmonic analysis. H\"ormander
subsequently recast the construction in terms of a preliminary
nonholomorphic B\'ezout solution followed by a
\(\bar\partial\)-correction organized through the Koszul complex
\cite{MR226387}. An important advantage of H\"ormander's approach was that it treated
the B\'ezout equation directly for an arbitrary finite number of
generators.

For countably many generators, Rosenblum \cite{MR570865} and
Tolokonnikov \cite{MR595742} independently proved that the classical
corona theorem remains valid under the normalization
\[
\delta^2
\leq
\sum_{j=1}^\infty |a_j(z)|^2
\leq1.
\]
In this setting the data may be regarded as the analytic row operator
\[
A(z):\ell^2\longrightarrow\C,
\qquad
A(z)c=\sum_{j=1}^\infty a_j(z)c_j,
\]
and a corona solution as a bounded analytic column
\(B=(b_j)_{j\geq1}^{\mathsf T}\) satisfying \(AB=1\).
Rosenblum's proof explicitly develops this operator-theoretic
viewpoint by reducing the problem to an estimate for an associated
Hilbert-space transformation \cite{MR570865}.

The classical corona problem also admits matrix-valued and
operator-valued extensions. The finite matrix case was studied by Fuhrmann
\cite{MR222701}, while Trent and Zhang developed matricial corona
theorems for more general function algebras
\cite{MR2213732,MR2317961}. Let \(E\) and \(E_*\) be complex Hilbert spaces, and let
\(\mathcal B(E,E_*)\) denote the bounded linear operators from
\(E\) to \(E_*\). Suppose that $
A\in H^\infty\bigl(\D,\mathcal B(E,E_*)\bigr)$ 
satisfies
\[
\delta^2 I_{E_*}
\leq
A(z)A(z)^*
\leq
I_{E_*},
\qquad z\in\D,
\]
where \(I_{E_*}\) denotes the identity operator on \(E_*\). 
The pointwise right inverse $A(z)^*\bigl(A(z)A(z)^*\bigr)^{-1}$ 
need not be holomorphic, whereas the matrix corona problem requires a
bounded analytic right inverse.  The matrix corona problem therefore asks
whether pointwise uniform surjectivity can be upgraded to the
existence of a bounded \emph{analytic} right inverse. The operator corona problem is closely connected with analytic
projections and the geometry of holomorphic vector bundles;
see \cite{TreilWick2009}.  

The distinction between the dimensions of \(E\) and \(E_*\) is
crucial in this problem. The unrestricted operator corona theorem
fails in general in the infinite-dimensional setting, whereas strong
positive results remain available when the target space is
finite-dimensional; see, for example, \cite{TreilWick2009}. In the form needed here, the theorem of
Treil and Wick states that, if $
\dim E_*<\infty$ 
and
\[
\delta^2 I_{E_*}
\leq
A(z)A(z)^*
\leq
I_{E_*},
\]
then the equation $
Av=g$ 
is solvable in the corresponding vector-valued \(H^p\)-space,
\(1\leq p\leq\infty\), with an estimate depending on \(\delta\) and
\(\dim E_*\), but not on \(\dim E\); see
\cite[Theorem~0.1]{MR2158178}. The independence from the source
dimension is precisely the feature that allows us to treat finite and
countably infinite quaternionic families by the same argument.

There is also a parallel Hilbert-space formulation of the corona
problem. If \(A\) is regarded as a multiplier from
\(H^2(\D,E)\) to \(H^2(\D,E_*)\), one may ask whether the
corresponding multiplication operator \(M_A\) satisfies a quantitative
lower bound $
M_A M_A^*\geq \eta^2 I$. 
The classical Toeplitz corona theorem relates such an operator
inequality to the existence of a bounded analytic multiplier right
inverse; see, for example,
\cite{MR383098,MR482355,MR3329539}. As described in \cite{MR3329539}, in the classical disk setting the
corona problem can be separated into an \(H^2\)-corona problem and an
operator-theoretic Toeplitz corona problem.  

We now turn to the quaternionic setting. Let
\[
\B:=\{q\in\Hq:|q|<1\},
\qquad
\Sph:=\{I\in\Hq:I^2=-1\}.
\]
Slice regular function theory provides a different extension of
one-variable complex analysis: instead of increasing the dimension
of the domain, one replaces the scalar field \(\C\) by the
noncommutative skew field \(\Hq\). As pointed out in \cite{MR5081054}, ideal and corona problems in the
hypercomplex setting are complicated by the fact that pointwise
multiplication does not preserve slice hyperholomorphicity.  Thus the appropriate
quaternionic B\'ezout equation is
\[
\sum_j f_j\star g_j=1,
\]
and the noncommutativity of \(\star\) prevents a direct componentwise
transfer of the classical complex theory. We use the standard
terminology and basic results from
\cite{GentiliStruppa2007,MR4501241,
ColomboGentiliSabadiniStruppa2009}.  For the operator-theoretic aspects of slice hyperholomorphic function
theory, we refer to \cite{MR3585855}. A systematic
treatment of quaternionic Hilbert spaces and the associated Hardy
spaces can be found in
\cite{MR4841413}.

The first step toward a quaternionic corona theory was an algebraic
one. Gentili, Sarfatti, and Struppa proved that finitely many slice
regular functions on a symmetric slice domain with no common zero
generate the whole ring of slice regular functions
\cite{MR3621101}. In other words, under the qualitative nonvanishing
assumption one can solve
\[
\sum_{j=1}^n f_j\star h_j=1
\]
with slice regular coefficients \(h_j\). This result shows that the
basic ideal-theoretic phenomenon survives in the noncommutative
setting, but it does not provide bounded solutions or quantitative
\(H^\infty\)-estimates.

The bounded problem requires substantially more information. The
finite-generator quaternionic \(H^\infty\)-corona theorem was
established in \cite{MR5081054}: a quantitative lower bound
\[
\delta^2
\leq
\sum_{j=1}^n|f_j(q)|^2
\leq1
\]
produces bounded slice regular solutions of
\[
\sum_{j=1}^n f_j\star g_j=1.
\]
The proof proceeds, after fixing a complex slice, through a coupled
system of complex B\'ezout equations and a determinant-type
condition involving the holomorphic splitting components. This
approach is effective for finitely many generators, but it leaves
several structural questions open. In particular, it is natural to
ask whether the dependence on the number of generators can be
removed, whether the determinant expression has a more intrinsic
operator-theoretic meaning, and whether the quaternionic corona
problem admits an \(H^2\) or Toeplitz formulation. The first and the
last of these questions were explicitly raised in
\cite[Section~4, items~\textup{(1)} and~\textup{(2)}]{MR5081054}.

The purpose of the present paper is to address these questions by means
of a fixed-slice matrix realization.  Fix \(I\in\Sph\), set
\[
\CI:=\R+\R I,
\qquad
\D_I:=\B\cap\CI,
\]
and choose \(J\in\Sph\) with \(J\perp I\).  If $
f_I=F+GJ$ 
is the splitting of a slice regular function \(f\), then, for a
\(\CI\)-valued holomorphic function \(U\), we write $
U^\#(z):=\overline{U(\overline z)}$ 
and define
\begin{equation}\label{eq:Phi-def-intro}
\RI(f)(z)
:=
\begin{pmatrix}
F(z)&G(z)\\
-G^\#(z)&F^\#(z)
\end{pmatrix}.
\end{equation}

The use of a $2\times 2$ complex matrix realization to encode the
quaternionic $\star$-product as ordinary matrix multiplication goes back, in the setting of quaternionic Wiener algebras, to Alpay,
Colombo, Kimsey, and Sabadini; see
\cite[Eq.~(2.6) and Lemma~2.4]{MR3554256}.
We subsequently employed this fixed-slice matrix viewpoint in our study
of slice regular composition operators on quaternionic Fock spaces
\cite{LinLuZu2026FockMatrix}.
The features of this realization that are crucial for the present work are the spherical operator-norm identity and its extension to finite and countable operator rows. These make it possible to transfer the pointwise quaternionic corona condition exactly to an operator-valued corona condition with fixed two-dimensional target space, leading to cardinality-independent $H^\infty$-estimates, $H^p$-solvability, and the Toeplitz and Leech factorization results below.

The present paper and the recent work of Colombo, Pozzi, Sabadini, and Wick \cite{ColomboPozziSabadiniWick2026MatrixCorona} were developed independently. Their work extends the matrix-Corona strategy to slice monogenic functions with values in general Clifford algebras and develops the associated Cauchy–Binet determinant structure. The overlap between the two papers is confined to the quaternionic matrix-realization/matrix-Corona mechanism. The directions of the two works are otherwise complementary: \cite{ColomboPozziSabadiniWick2026MatrixCorona} emphasizes the Clifford-algebra extension, whereas the present paper focuses on countably many quaternionic generators, cardinality-independent estimates, $H^p$-solvability, and the operator-theoretic Leech and Toeplitz corona formulations.

The usefulness of this realization comes from three compatible
features. First, it converts the \(\star\)-product of slice regular
functions into ordinary multiplication of analytic matrix-valued
functions:
\[
\RI(f\star g)=\RI(f)\RI(g).
\]
Second, its range is a fixed-point real form of an algebra of
\(2\times2\) analytic matrix-valued functions. Third, and most
importantly for the corona problem, the matrix norm records all
values of \(f\) on the quaternionic sphere through a point of the
fixed slice.

The fixed-slice realization converts the quaternionic pointwise
corona condition into an operator-valued corona condition on
\(\D_I\) with fixed two-dimensional target space. This allows us to
apply the operator-valued corona theorem of Treil and Wick
\cite[Theorem~0.1]{MR2158178}. The fixed-point symmetry of the
realization then recovers quaternionic solutions from the complex
operator solutions. Since the target dimension is always \(2\), the
same argument applies to both finite and countable families, with
estimates independent of the number of generators.

For \(0<\delta<1\), set
\begin{equation}\label{eq:explicit-constant}
  \delta_0:=\min\{\delta,\e^{-1/2}\},
  \qquad
  \kappa:=\sqrt{1+\e^2}+\sqrt{\e}+\sqrt{2}\,\e.
\end{equation}
Define
\begin{equation}\label{eq:C-def}
  \mathcal C_0(\delta)
  :=
  \frac{\kappa}{\delta_0^3}
  \log\frac{1}{\delta_0^4}
  +\frac{1}{\delta_0},
  \qquad
  \mathcal C(\delta)
  :=\sqrt2\,\mathcal C_0(\delta).
\end{equation}

Throughout the remainder of the paper, \(\Lambda\) denotes either
\[
\Lambda=\{1,\ldots,n\}
\qquad\text{or}\qquad
\Lambda=\mathbb N.
\]

Our first main result is a quaternionic \(H^\infty\)-corona theorem
with estimates independent of the cardinality of the family of
generators.

\begin{theorem}[Quaternionic \(H^\infty\)-corona theorem]
\label{thm:quat-corona}
Let \(0<\delta<1\),  and let \(f_j\in H^\infty(\B)\), \(j\in\Lambda\),
satisfy
\begin{equation}\label{eq:quat-corona-hyp}
  \delta^2
  \leq
  \sum_{j\in\Lambda}|f_j(q)|^2
  \leq1,
  \qquad q\in\B.
\end{equation}
Then there exist $
g_j\in H^\infty(\B), j\in\Lambda$, 
such that
\begin{equation}\label{eq:quat-corona-conclusion}
  \sum_{j\in\Lambda}f_j\star g_j=1
  \qquad\text{on }\B.
\end{equation}
If \(\Lambda=\mathbb N\), the series in
\eqref{eq:quat-corona-conclusion} converges absolutely and uniformly
on compact subsets of \(\B\).

Moreover,
\begin{equation}\label{eq:quat-corona-l2-bound}
  \sup_{q\in\B}
  \left(
    \sum_{j\in\Lambda}|g_j(q)|^2
  \right)^{1/2}
  \leq\mathcal C(\delta).
\end{equation}
Consequently,
\begin{equation}\label{eq:quat-corona-explicit-bound}
  \sup_{j\in\Lambda}\|g_j\|_{H^\infty(\B)}
  \leq\mathcal C(\delta)
  \lesssim
  \delta^{-3}\log\frac{\e}{\delta},
\end{equation}
where the implied constant is absolute and independent of the
cardinality of \(\Lambda\).
\end{theorem}

In particular, Theorem~\ref{thm:quat-corona} extends the
finite-generator result of \cite{MR5081054} to countably many
generators, with a bound independent of the cardinality of the
family.

Our second main result gives Hardy-space solvability for the full range
\(1\leq p<\infty\).  We denote by \(H^p(\B)\) the quaternionic Hardy
space and by $
H^p(\B;\ell^2(\Lambda))$ 
the corresponding \(\ell^2(\Lambda)\)-valued Hardy space; the precise
definitions are recalled in Section~\ref{sec:Hp-Toeplitz-corona}.  For \(1\leq p<\infty\), set $
\rho_p:=2^{\max\{1/2,\,1/p\}}$.

\begin{theorem}[Quaternionic \(H^p\)-corona theorem]
\label{thm:quaternionic-Hp-corona}
Let \(0<\delta<1\), let \(1\leq p<\infty\), and let
\(f_j\in H^\infty(\B)\), \(j\in\Lambda\), satisfy
\begin{equation}\label{eq:Hp-corona-hyp}
\delta^2
\leq
\sum_{j\in\Lambda}|f_j(q)|^2
\leq1,
\qquad q\in\B.
\end{equation}
Then, for every \(k\in H^p(\B)\), there exists $
\mathbf h=(h_j)_{j\in\Lambda}
\in H^p(\B;\ell^2(\Lambda))$ 
such that
\begin{equation}\label{eq:Hp-corona-equation}
\sum_{j\in\Lambda}f_j\star h_j=k
\end{equation}
and
\begin{equation}\label{eq:Hp-corona-estimate}
\|\mathbf h\|_{H^p(\B;\ell^2(\Lambda))}
\leq
\rho_p^2
\mathcal C_0(\delta)
\|k\|_{H^p(\B)}.
\end{equation}
If \(\Lambda=\mathbb N\), the series in
\eqref{eq:Hp-corona-equation} converges in \(H^p(\B)\).
\end{theorem}

We next turn to the operator-theoretic form of the corona problem.
Let \(H^2(\B)\) be the quaternionic Hardy space corresponding to
\(p=2\) in \eqref{eq:quaternionic-Hp-norm}, and set  $
E_\Lambda:=\ell^2(\Lambda;\CI^2)$. 
For \(f\in H^\infty(\B)\), define the left multiplier
\[
M_fh:=f\star h,
\qquad
h\in H^2(\B).
\]
For $
\mathbf f=(f_j)_{j\in\Lambda}$, 
let $
\mathcal M_{\mathbf f}:
H^2(\B;\ell^2(\Lambda))
\longrightarrow
H^2(\B)$ 
denote the associated row multiplier,
\begin{equation}\label{eq:row-multiplier-def}
\mathcal M_{\mathbf f}(h_j)_{j\in\Lambda}
:=
\sum_{j\in\Lambda}f_j\star h_j.
\end{equation}
For countable \(\Lambda\), the boundedness and \(H^2\)-norm convergence
of this series under the hypothesis below are established in
Section~\ref{sec:Hp-Toeplitz-corona}.

Our third main result is the following Toeplitz-corona theorem.

\begin{theorem}[Quaternionic Toeplitz corona theorem]
\label{thm:quaternionic-toeplitz-corona}
Let $
f_j\in H^\infty(\B),
j\in\Lambda$, 
satisfy
\[
\sum_{j\in\Lambda}|f_j(q)|^2
\leq1,
\qquad q\in\B,
\]
and let $\eta>0$. 
Then the following statements are equivalent.

\begin{enumerate}
\item[\textup{(i)}]
The Toeplitz lower bound
\begin{equation}\label{eq:full-toeplitz-lower}
\mathcal M_{\mathbf f}
\mathcal M_{\mathbf f}^*
\geq
\eta^2I_{H^2(\B)}
\end{equation}
holds.

\item[\textup{(ii)}]
There exist $
g_j\in H^\infty(\B),
j\in\Lambda$, 
such that
\begin{equation}\label{eq:full-toeplitz-bezout}
\sum_{j\in\Lambda}
f_j\star g_j
=
1,
\end{equation}
and
\[
\Psi_{\mathbf g}(z)
=
\begin{pmatrix}
\RI(g_1)(z)\\
\RI(g_2)(z)\\
\vdots
\end{pmatrix}
\in
H^\infty
\bigl(
\D_I,\mathcal B(\CI^2,E_\Lambda)
\bigr)
\]
satisfies
\begin{equation}\label{eq:full-toeplitz-Psi-bound}
\|\Psi_{\mathbf g}\|_\infty
\leq
\eta^{-1}.
\end{equation}
\end{enumerate}

If \(\Lambda=\mathbb N\), the series in
\eqref{eq:full-toeplitz-bezout}
converges absolutely and uniformly on compact subsets of \(\B\).

\end{theorem}

Theorem~\ref{thm:quaternionic-toeplitz-corona} is obtained in
Section~\ref{sec:Hp-Toeplitz-corona} as the special case \(h\equiv1\) of a more general
quaternionic Leech factorization theorem.  More precisely, we obtain
a Toeplitz-operator characterization, with quantitative control, of
factorizations of the form
\[
h=\sum_{j\in\Lambda}f_j\star g_j,
\qquad
h\in H^\infty(\B).
\]
The proof combines a unitary fixed-slice realization of quaternionic
\(H^2\)-multipliers with the classical Leech factorization theorem and
the fixed-point symmetry of the matrix realization.

The paper is organized as follows.  Section~2 recalls the basic
properties of slice regular functions.  Section~3 develops the
fixed-slice matrix realization, proves its multiplicativity, and
establishes the spherical operator-norm identity together with its
extension to finite and countable operator rows.  In Section~4 we
prove a fixed-point operator corona theorem by combining the
operator-valued corona theorem of Treil and Wick with the symmetry of
the fixed-slice realization.  Section~5 proves
Theorem~\ref{thm:quat-corona} and identifies the determinant condition
in the finite-generator theorem of \cite{MR5081054} with the
Cauchy--Binet expansion of the corresponding Gram determinant.
Finally, Section~6 proves
Theorem~\ref{thm:quaternionic-Hp-corona}, develops the Hilbert-space
multiplier realization, establishes the quaternionic Leech
factorization theorem, and derives
Theorem~\ref{thm:quaternionic-toeplitz-corona}.

\section{Preliminaries}\label{sec:preliminaries}

The quaternion skew field is
\[
\Hq
=
\{x_0+x_1\mathrm i+x_2\mathrm j+x_3\mathrm k:
x_0,x_1,x_2,x_3\in\R\},
\]
with
\[
\mathrm i^2=\mathrm j^2=\mathrm k^2=-1,
\qquad
\mathrm i\mathrm j=-\mathrm j\mathrm i=\mathrm k,
\]
and the corresponding cyclic relations. For
\(q=x_0+x_1\mathrm i+x_2\mathrm j+x_3\mathrm k\), set
\[
\overline q=x_0-x_1\mathrm i-x_2\mathrm j-x_3\mathrm k,
\qquad
|q|=(q\overline q)^{1/2}.
\]

For \(I\in\Sph\), the plane \(\CI=\R+\R I\) is a copy of the complex
plane and \(\D_I=\B\cap\CI\) is its unit disk. For \(I,J\in\Sph\), we write \(J\perp I\) when \(I\) and \(J\)
are orthogonal; equivalently, \(IJ=-JI\).

We recall the standard definition; see
\cite[Definition~1.1]{MR4501241}.
\begin{definition}
A real-differentiable function \(f:\B\to\Hq\) is \emph{left slice
regular} if, for every \(I\in\Sph\),
\begin{equation}\label{eq:slice-CR}
\overline\partial_I f_I(x+Iy)
:=
\frac12\left(
\frac{\partial}{\partial x}
+I\frac{\partial}{\partial y}
\right)f_I(x+Iy)=0,
\end{equation}
where \(f_I=f|_{\D_I}\).
\end{definition}

A function is slice regular on \(\B\) if and only if it has a
power-series expansion
\begin{equation}\label{eq:slice-power-series}
f(q)=\sum_{m=0}^\infty q^m a_m,
\qquad a_m\in\Hq,
\end{equation}
converging uniformly on compact subsets of \(\B\); see
\cite[Theorems~1.6 and~1.10]{MR4501241}.

We write \(H^\infty(\D_I,\CI)\) for the bounded
\(\CI\)-valued holomorphic functions on \(\D_I\).
 \(I_E\) denotes the identity
operator on a Hilbert space \(E\). In particular, \(I_2\) denotes
the identity operator on \(\CI^2\).

\begin{lemma}[Splitting and extension lemma]
\label{lem:splitting-extension}
Let \(I,J\in\Sph\) with \(J\perp I\).
\begin{enumerate}
\item[\textup{(i)}]
If \(f\) is slice regular on \(\B\), then there are unique holomorphic
functions \(F,G:\D_I\to\CI\) such that
\begin{equation}\label{eq:splitting}
f_I(z)=F(z)+G(z)J.
\end{equation}
Moreover,
\begin{equation}\label{eq:splitting-norm}
|f_I(z)|^2=|F(z)|^2+|G(z)|^2.
\end{equation}
\item[\textup{(ii)}]
If \(h=F+GJ\), where \(F,G\in\mathcal O(\D_I,\CI)\), then
\begin{equation}\label{eq:extension-formula}
\operatorname{ext}_I(h)(x+Ly)
:=
\frac12(1-LI)h(x+Iy)
+\frac12(1+LI)h(x-Iy)
\end{equation}
defines the unique slice regular extension of \(h\) to \(\B\). If
\(h\) is bounded, then
\begin{equation}\label{eq:extension-bound}
\sup_{q\in\B}|\operatorname{ext}_I(h)(q)|
\leq2\sup_{z\in\D_I}|h(z)|.
\end{equation}
\end{enumerate}
These statements are the splitting and extension lemmas,
respectively; see \cite[Lemmas~1.3 and~1.21]{MR4501241}.
\end{lemma}

If
\[
f(q)=\sum_{m=0}^\infty q^ma_m,
\qquad
g(q)=\sum_{m=0}^\infty q^mb_m,
\]
their slice regular product is
\begin{equation}\label{eq:star-product-definition}
(f\star g)(q)
:=
\sum_{m=0}^\infty q^m
\left(\sum_{k=0}^m a_kb_{m-k}\right).
\end{equation}
The product is unital and associative; see \cite{MR4501241}. The
Banach-algebra norm estimate follows from the isometric matrix
realization in Proposition~\ref{prop:fixed-slice-realization} below.

For \(1\leq p<\infty\), the quaternionic Hardy space \(H^p(\B)\)
consists of all slice regular functions \(h\) on \(\B\) such that
\begin{equation}\label{eq:quaternionic-Hp-norm}
\|h\|_{H^p(\B)}
:=
\sup_{L\in\Sph}\sup_{0<r<1}
\left(
\frac1{2\pi}
\int_0^{2\pi}
|h(re^{L\theta})|^p\,d\theta
\right)^{1/p}
<\infty.
\end{equation}
See, for example, \cite{MR4841413,MR3585855}.

The bounded slice regular functions form the space
\[
H^\infty(\B)
:=
\{f:\B\to\Hq:f\text{ is slice regular and }
\|f\|_{H^\infty(\B)}<\infty\},
\]
where \(\|f\|_{H^\infty(\B)}=\sup_{q\in\B}|f(q)|\).

For a family
\[
\mathbf h=(h_j)_{j\in\Lambda},
\]
we set
\begin{equation}\label{eq:Hp-sequence-norm}
\|\mathbf h\|_{H^p(\B;\ell^2(\Lambda))}
:=
\sup_{L\in\Sph}\sup_{0<r<1}
\left(
\frac1{2\pi}
\int_0^{2\pi}
\left(
\sum_{j\in\Lambda}
|h_j(re^{L\theta})|^2
\right)^{p/2}
d\theta
\right)^{1/p}.
\end{equation}

For a complex Hilbert space \(E\), we use the notation
\(H^p(\D_I,E)\) for the corresponding \(E\)-valued Hardy space.

For \(p=2\), the norm in
\eqref{eq:quaternionic-Hp-norm} agrees with the coefficient norm
\[
\|h\|_{H^2(\B)}^2
=
\sum_{m=0}^\infty |a_m|^2,
\qquad
h(q)=\sum_{m=0}^\infty q^m a_m.
\]
We use the right quaternionic inner product
\[
\langle h,u\rangle_{H^2(\B)}
=
\sum_{m=0}^\infty \overline{a_m}b_m,
\qquad
u(q)=\sum_{m=0}^\infty q^m b_m.
\]
For a sequence
\(\mathbf h=(h_j)_{j\in\Lambda}\),
\begin{equation}\label{eq:sequence-H2-norm}
\|\mathbf h\|_{H^2(\B;\ell^2(\Lambda))}^2
=
\sum_{j\in\Lambda}\|h_j\|_{H^2(\B)}^2.
\end{equation}
See also \cite{MR4841413,MR3585855}.

For a \(\CI\)-valued holomorphic function \(U\) on \(\D_I\), reflected
conjugation is \(U^\#(z)=\overline{U(\overline z)}\). For a matrix or
vector-valued holomorphic function, \(\#\) is applied entrywise. This
is a conjugate-linear isometric involution.

For Hilbert spaces \(E\) and \(E_*\), the space
\(H^\infty\bigl(\D_I,\mathcal B(E,E_*)\bigr)\)
consists of bounded operator-valued holomorphic functions, with norm
\[
\|A\|_\infty
:=
\sup_{z\in\D_I}\|A(z)\|_{\op}.
\]
When the domain is finite-dimensional, \(\|A\|_{\mathrm{HS}}\)
denotes the Hilbert--Schmidt norm.

\section{The fixed-slice matrix realization}
\label{sec:fixed-slice-realization}

Fix \(I,J\in\Sph\) with \(J\perp I\). Every quaternion
\(a\in\Hq\) has a unique representation
\[
a=\alpha+\beta J,
\qquad \alpha,\beta\in\CI.
\]
Set
\begin{equation}\label{eq:chi-def}
\chi_I(a)
:=
\begin{pmatrix}
\alpha&\beta\\
-\overline\beta&\overline\alpha
\end{pmatrix}.
\end{equation}
Then \(\chi_I:\Hq\to M_2(\CI)\) is an injective unital
homomorphism of real algebras; see \cite{Zhang1997}.

\subsection{The  fixed-point space}

Define
\[
H^\infty_\Theta(\D_I,M_2(\CI))
:=
\left\{
\begin{pmatrix}
F&G\\
-G^\#&F^\#
\end{pmatrix}:
F,G\in H^\infty(\D_I,\CI)
\right\}.
\]
If \(f_I=F+GJ\) is the splitting of a slice regular function, its
fixed-slice matrix realization is the matrix
\(\RI(f)\) defined in \eqref{eq:Phi-def-intro}. 
Equivalently, if
\[
f(q)=\sum_{m\geq0}q^ma_m,
\qquad a_m=\alpha_m+\beta_mJ,
\]
then
\begin{equation}\label{eq:Phi-coeff}
\RI(f)(z)=\sum_{m\geq0}z^m\chi_I(a_m).
\end{equation}

This realization is the slice-regular quaternionic counterpart of the
matrix map introduced for quaternionic Wiener algebras in
\cite[Eq.~(2.6)]{MR3554256}; its
multiplicative property corresponds to
\cite[Lemma~2.4]{MR3554256}.

\begin{proposition} 
\label{prop:fixed-slice-realization}
The map $
\RI:H^\infty(\B)
\longrightarrow
H^\infty_\Theta(\D_I,M_2(\CI))$ 
is a real-linear bijection and satisfies
\begin{equation}\label{eq:RI-multiplicative}
\RI(f\star g)=\RI(f)\RI(g),
\qquad f,g\in H^\infty(\B).
\end{equation}

Moreover, let \(f\) be slice regular on \(\B\), let
\(z=x+Iy\in\D_I\), and let
\[
\xi=(a,b)\in\CI^2,
\qquad |a|^2+|b|^2=1.
\]
Set
\[
u:=a+bJ,
\qquad L:=u^{-1}Iu.
\]
Then \(|u|=1\), \(L\in\Sph\), and
\begin{equation}\label{eq:spherical-row-formula}
\|\xi\RI(f)(z)\|_{\CI^2}
=
|f(x+Ly)|.
\end{equation}
Consequently,
\begin{equation}\label{eq:spherical-op-norm}
\|\RI(f)(x+Iy)\|_{\op}
=
\max_{L\in\Sph}|f(x+Ly)|.
\end{equation}
In particular,
\[
\|\RI(f)\|_\infty
=
\|f\|_{H^\infty(\B)},
\]
and hence \(\RI\) is an isometric isomorphism of unital real Banach
algebras.
\end{proposition}

\begin{proof}
By Lemma~\ref{lem:splitting-extension}\textup{(i)}, every
\(f\in H^\infty(\B)\) has a unique splitting
\(f_I=F+GJ\). By \eqref{eq:splitting-norm},
\(F,G\in H^\infty(\D_I,\CI)\) and the correspondence \(f\mapsto\RI(f)\) is real linear.
Conversely, if
\[
\Psi=
\begin{pmatrix}
F&G\\
-G^\#&F^\#
\end{pmatrix}
\in H^\infty_\Theta(\D_I,M_2(\CI)),
\]
then $
f=\operatorname{ext}_I(F+GJ)$ 
belongs to \(H^\infty(\B)\) by
Lemma~\ref{lem:splitting-extension}\textup{(ii)}, and
\(\RI(f)=\Psi\). Hence \(\RI\) is a real-linear bijection.

Let
\[
f(q)=\sum_{m\ge0}q^ma_m,
\qquad
g(q)=\sum_{m\ge0}q^mb_m.
\]
Using \eqref{eq:Phi-coeff}, the multiplicativity of \(\chi_I\), and
the Cauchy product, we obtain
\[
\begin{aligned}
\RI(f)(z)\RI(g)(z)
&=
\sum_{m\ge0}z^m
\sum_{k=0}^m\chi_I(a_k)\chi_I(b_{m-k})\\
&=
\sum_{m\ge0}z^m
\chi_I\left(\sum_{k=0}^ma_kb_{m-k}\right)
=
\RI(f\star g)(z).
\end{aligned}
\]
Since the matrix product on the left is bounded, its first-row
entries are bounded. The extension lemma therefore shows that
\(f\star g\in H^\infty(\B)\). This proves
\eqref{eq:RI-multiplicative}.

Now write
\[
f(q)=\sum_{m\ge0}q^mc_m,
\qquad
c_m=\alpha_m+\beta_mJ.
\]
For a unit row vector \(\xi=(a,b)\), put \(u=a+bJ\). Then
\(|u|=1\), and \(\xi\) is the first row of \(\chi_I(u)\).
Hence, by \eqref{eq:Phi-coeff} and the multiplicativity of
\(\chi_I\), the entries of \(\xi\RI(f)(z)\) are the two
\(\CI\)-components of
\[
\sum_{m\ge0}z^muc_m.
\]
Since
\[
\sum_{m\ge0}z^muc_m
=
u\sum_{m\ge0}(u^{-1}zu)^mc_m
=
u f(u^{-1}zu),
\]
and
\[
u^{-1}zu=x+(u^{-1}Iu)y=x+Ly,
\]
we obtain
\[
\|\xi\RI(f)(z)\|_{\CI^2}
=
|u f(x+Ly)|
=
|f(x+Ly)|.
\]
This proves \eqref{eq:spherical-row-formula}.

As \(u\mapsto u^{-1}Iu\) maps the unit quaternions onto \(\Sph\),
taking the supremum over all unit row vectors \(\xi\) gives
\eqref{eq:spherical-op-norm}. Therefore
\[
\begin{aligned}
\|\RI(f)\|_\infty
&=
\sup_{x^2+y^2<1}
\max_{L\in\Sph}|f(x+Ly)| =
\sup_{q\in\B}|f(q)|
=
\|f\|_{H^\infty(\B)}.
\end{aligned}
\]
Thus \(\RI\) is isometric. Since
\(H^\infty_\Theta(\D_I,M_2(\CI))\) is a closed subspace of
\(H^\infty(\D_I,M_2(\CI))\), it is complete; together with
\eqref{eq:RI-multiplicative} and \(\RI(1)=I_2\), this proves that
\(\RI\) is an isometric isomorphism of unital real Banach algebras.
\end{proof}

Let
\[
J_2:=
\begin{pmatrix}
0&1\\
-1&0
\end{pmatrix}.
\]
For \(A\in H^\infty(\D_I,M_2(\CI))\), define
\begin{equation}\label{eq:Theta-22}
\Theta_{2,2}(A):=J_2A^\#J_2^{-1}.
\end{equation}
A direct calculation gives
\begin{equation}\label{eq:fixed-point-range}
A\in H^\infty_\Theta(\D_I,M_2(\CI))
\quad\Longleftrightarrow\quad
\Theta_{2,2}(A)=A.
\end{equation}

Set
\begin{equation}\label{eq:E-Lambda-def}
E_\Lambda:=\ell^2(\Lambda;\CI^2),
\qquad
\mathcal J_\Lambda:=\bigoplus_{j\in\Lambda}J_2.
\end{equation}
For finite \(\Lambda=\{1,\ldots,n\}\), this means
\(E_\Lambda=\CI^{2n}\) and
\(\mathcal J_\Lambda=\diag(J_2,\ldots,J_2)\).

Let \(C_\Lambda\) and \(C_2\) denote coordinatewise conjugation on
\(E_\Lambda\) and \(\CI^2\), respectively. For
\[
A\in H^\infty
\bigl(\D_I,\mathcal B(E_\Lambda,\CI^2)\bigr),
\]
define
\[
A^\#(z):=C_2A(\overline z)C_\Lambda,
\qquad
\Theta_{2,\Lambda}(A)
:=J_2A^\#\mathcal J_\Lambda^{-1}.
\]
Similarly, for
\[
B\in H^\infty
\bigl(\D_I,\mathcal B(\CI^2,E_\Lambda)\bigr),
\]
define
\[
B^\#(z):=C_\Lambda B(\overline z)C_2,
\qquad
\Theta_{\Lambda,2}(B)
:=\mathcal J_\Lambda B^\#J_2^{-1}.
\]
These maps are conjugate-linear isometric involutions and satisfy
\begin{equation}\label{eq:Theta-product}
\Theta_{2,2}(AB)
=
\Theta_{2,\Lambda}(A)\Theta_{\Lambda,2}(B).
\end{equation} 

\subsection{Finite and countable operator rows}

\begin{lemma} 
\label{lem:operator-row}
Let \(f_j\in H^\infty(\B)\), \(j\in\Lambda\),
and suppose that
\[
\sum_{j\in\Lambda}|f_j(q)|^2\leq1,
\qquad q\in\B.
\]
Set $
\Phi_j:=\RI(f_j)$, $j\in\Lambda$. 
Then
\begin{equation}\label{eq:operator-row-def}
\Phi(z)(v_j)_{j\in\Lambda}
:=
\sum_{j\in\Lambda}\Phi_j(z)v_j
\end{equation}
defines a function $
\Phi\in H^\infty
\bigl(\D_I,\mathcal B(E_\Lambda,\CI^2)\bigr)$ 
with $
\|\Phi\|_\infty\leq1$. 
Moreover,
\begin{equation}\label{eq:row-fixed-point}
\Theta_{2,\Lambda}(\Phi)=\Phi.
\end{equation}
\end{lemma}

\begin{proof}
For every finite subset \(F\subset\Lambda\), let $
\Phi_F=(\Phi_j)_{j\in F}$. 
If \(z=x+Iy\) and \(\xi=(a,b)\in\CI^2\) is a unit row vector, set
\[
u:=a+bJ,
\qquad
L:=u^{-1}Iu.
\]
By \eqref{eq:spherical-row-formula},
\[
\begin{aligned}
\xi\Phi_F(z)\Phi_F(z)^*\xi^*
&=
\sum_{j\in F}\|\xi\Phi_j(z)\|_{\CI^2}^2  =
\sum_{j\in F}|f_j(x+Ly)|^2
\leq1.
\end{aligned}
\]
Hence $
\|\Phi_F(z)\|_{\op}\leq1$
 uniformly in \(F\) and \(z\).

It follows that, for every \(\eta\in\CI^2\),
\[
\sum_{j\in\Lambda}\|\Phi_j(z)^*\eta\|_{\CI^2}^2
\leq\|\eta\|_{\CI^2}^2.
\]
Thus $
\Phi(z)^*\eta
:=
\bigl(\Phi_j(z)^*\eta\bigr)_{j\in\Lambda}$ 
defines a contraction from \(\CI^2\) into \(E_\Lambda\). Its adjoint
is precisely the operator \(\Phi(z)\) in
\eqref{eq:operator-row-def}, and therefore $
\|\Phi(z)\|_{\op}\leq1$.

For each \(v=(v_j)_{j\in\Lambda}\in E_\Lambda\), the partial sums in
\eqref{eq:operator-row-def} converge uniformly on \(\D_I\). Indeed,
for every finite subset \(F\subset\Lambda\),
\[
\left\|
\sum_{j\in F}\Phi_j(z)v_j
\right\|_{\CI^2}
\leq
\left(
\sum_{j\in F}\|v_j\|_{\CI^2}^2
\right)^{1/2},
\qquad z\in\D_I.
\]
If \(\Lambda=\mathbb N\), then, for every finite subset
\(F\subset\{N+1,N+2,\ldots\}\),
\[
\sup_{z\in\D_I}
\left\|
\sum_{j\in F}\Phi_j(z)v_j
\right\|_{\CI^2}
\leq
\left(
\sum_{j>N}\|v_j\|_{\CI^2}^2
\right)^{1/2}
\longrightarrow0
\qquad (N\to\infty),
\]
because \(v\in E_\Lambda=\ell^2(\Lambda;\CI^2)\). Thus the series in
\eqref{eq:operator-row-def} is uniformly Cauchy on \(\D_I\), and hence
converges uniformly there. The finite case is immediate. Consequently,
\(z\mapsto\Phi(z)v\) is holomorphic.

Therefore the two coordinate rows of \(\Phi\) are weakly holomorphic
\(E_\Lambda^*\)-valued functions. Since they are uniformly bounded,
they are norm-holomorphic. Consequently,
\[
\Phi\in H^\infty
\bigl(\D_I,\mathcal B(E_\Lambda,\CI^2)\bigr),
\qquad
\|\Phi\|_\infty\leq1.
\]

Finally, each block \(\Phi_j=\RI(f_j)\) is fixed by
\(\Theta_{2,2}\). Hence the block row \(\Phi\) satisfies
\eqref{eq:row-fixed-point}.
\end{proof}

\begin{lemma} 
\label{lem:spherical-row}
Let \((f_j)_{j\in\Lambda}\) satisfy the hypotheses of
Lemma~\ref{lem:operator-row}, and let \(\Phi\) be the corresponding
operator row. If
\[
z=x+Iy\in\D_I,
\qquad
\xi=(a,b)\in\CI^2,
\qquad
\|\xi\|_{\CI^2}=1,
\]
and
\[
u:=a+bJ,
\qquad
L:=u^{-1}Iu,
\]
then
\begin{equation}\label{eq:Lambda-row-identity}
\xi\Phi(z)\Phi(z)^*\xi^*
=
\sum_{j\in\Lambda}|f_j(x+Ly)|^2.
\end{equation}
Consequently, for \(0<\delta<1\),
\[
\delta^2
\leq
\sum_{j\in\Lambda}|f_j(q)|^2
\leq1,
\qquad q\in\B,
\]
if and only if
\begin{equation}\label{eq:operator-corona-equivalence}
\delta^2I_2
\leq
\Phi(z)\Phi(z)^*
\leq I_2,
\qquad z\in\D_I.
\end{equation}
\end{lemma}

\begin{proof}
Summing \eqref{eq:spherical-row-formula} over finite subsets
\(F\subset\Lambda\) and then passing to \(F\uparrow\Lambda\) yields
\eqref{eq:Lambda-row-identity}. 
Suppose first that the pointwise condition in the statement of the
lemma holds.  
For every unit row vector \(\xi\), identity
\eqref{eq:Lambda-row-identity} gives $
\delta^2
\leq
\xi\Phi(z)\Phi(z)^*\xi^*
\leq1$. 
Since \(\Phi(z)\Phi(z)^*\) is a positive \(2\times2\) matrix, the
quadratic-form characterization of positive-operator inequalities
gives \eqref{eq:operator-corona-equivalence}.

Conversely, suppose that \eqref{eq:operator-corona-equivalence}
holds. Given \(q=x+Ly\in\B\), choose a unit quaternion \(u\) such that $
L=u^{-1}Iu$, 
and write \(u=a+bJ\) with \(a,b\in\CI\). Then
\(\xi=(a,b)\) is a unit row vector, and
\eqref{eq:Lambda-row-identity} gives
\[
\sum_{j\in\Lambda}|f_j(q)|^2
=
\xi\Phi(z)\Phi(z)^*\xi^*.
\]
Applying \eqref{eq:operator-corona-equivalence} to \(\xi\) yields
\[
\delta^2
\leq
\sum_{j\in\Lambda}|f_j(q)|^2
\leq1.
\]
This proves the equivalence.
\end{proof}

\section{A fixed-point operator corona theorem}
\label{sec:fixed-point-operator-corona}

\begin{theorem} 
\label{thm:real-matrix-corona}
Let \(0<\delta<1\), and suppose that $
\Phi\in H^\infty
\bigl(\D_I,\mathcal B(E_\Lambda,\CI^2)\bigr)$ 
satisfies $
\Theta_{2,\Lambda}(\Phi)=\Phi$
 and
\begin{equation}\label{eq:real-matrix-corona-hyp}
\delta^2I_2
\leq
\Phi(z)\Phi(z)^*
\leq I_2,
\qquad z\in\D_I.
\end{equation}
Then there exists $
\Psi\in H^\infty
\bigl(\D_I,\mathcal B(\CI^2,E_\Lambda)\bigr)$ 
such that
\begin{equation}\label{eq:real-matrix-bezout}
\Phi(z)\Psi(z)=I_2,
\qquad z\in\D_I,
\end{equation}
and
\begin{equation}\label{eq:Psi-fixed-point}
\Theta_{\Lambda,2}(\Psi)=\Psi.
\end{equation}
Moreover,
\begin{equation}\label{eq:real-matrix-explicit-bound}
\sup_{z\in\D_I}\|\Psi(z)\|_{\mathrm{HS}}
\leq\mathcal C(\delta).
\end{equation}
\end{theorem}

\begin{proof}
We identify the complex Hilbert space \(\CI\) isometrically with
\(\C\). Since
\[
\delta_0^2I_2
\leq
\Phi(z)\Phi(z)^*
\leq I_2
\]
and \(\delta_0^2\leq\e^{-1}\), apply
\cite[Theorem~0.1]{MR2158178} with
\[
E=E_\Lambda,
\qquad E_*=\C^2,
\qquad r=\dim E_*=2,
\qquad p=\infty.
\]
Applying the theorem to the constant function
\(g\equiv e_1=(1,0)^{\mathsf T}\) yields
\[
\omega\in H^\infty(\D_I,E_\Lambda),
\qquad
\Phi\omega=e_1,
\]
and
\begin{equation}\label{eq:omega-bound}
\|\omega\|_\infty\leq\mathcal C_0(\delta).
\end{equation}

The fixed-point identity for \(\Phi\) is equivalent to
\begin{equation}\label{eq:Phi-J-identity}
\Phi\mathcal J_\Lambda=J_2\Phi^\#.
\end{equation}
For an \(E_\Lambda\)-valued function, put
\(\omega^\#(z)=C_\Lambda\omega(\overline z)\), and define
\[
\widetilde\omega:=-\mathcal J_\Lambda\omega^\#,
\qquad
\Psi:=\bigl(\omega\ \widetilde\omega\bigr).
\]
Using \eqref{eq:Phi-J-identity}, we obtain
\[
\begin{aligned}
\Phi\widetilde\omega
&=-\Phi\mathcal J_\Lambda\omega^\# =-J_2\Phi^\#\omega^\# =-J_2(\Phi\omega)^\# =-J_2e_1=e_2.
\end{aligned}
\]
Hence \(\Phi\Psi=I_2\). Since
\(\widetilde\omega^\#=-\mathcal J_\Lambda\omega\), a direct
calculation, using \(\mathcal J_\Lambda^2=-I\) and \(J_2^2=-I_2\),
gives
\[
\mathcal J_\Lambda\Psi^\#J_2^{-1}=\Psi,
\]
which is \eqref{eq:Psi-fixed-point}.

Finally,
\[
\begin{aligned}
\|\Psi(z)\|_{\mathrm{HS}}^2
&=\|\omega(z)\|_{E_\Lambda}^2
+\|\widetilde\omega(z)\|_{E_\Lambda}^2\\
&=\|\omega(z)\|_{E_\Lambda}^2
+\|\omega(\overline z)\|_{E_\Lambda}^2\\
&\leq2\mathcal C_0(\delta)^2
=\mathcal C(\delta)^2.
\end{aligned}
\]
This proves \eqref{eq:real-matrix-explicit-bound}.
\end{proof}

\section{Quaternionic \(H^\infty\)-corona theorem}
\label{sec:quaternionic-Hinfty-corona}

\begin{proof}[Proof of Theorem~\ref{thm:quat-corona}]
Fix \(I,J\in\Sph\) with \(J\perp I\). For \(j\in\Lambda\), let $
\Phi_j:=\RI(f_j)$, 
and let \(\Phi\) be the operator row defined in
\eqref{eq:operator-row-def}. By Lemmas~\ref{lem:operator-row} and~\ref{lem:spherical-row},
the operator row \(\Phi\) belongs to
\(H^\infty(\D_I,\mathcal B(E_\Lambda,\CI^2))\), satisfies
\(\Theta_{2,\Lambda}(\Phi)=\Phi\), and
\eqref{eq:operator-corona-equivalence} holds.

Theorem~\ref{thm:real-matrix-corona} gives $
\Psi\in H^\infty
\bigl(\D_I,\mathcal B(\CI^2,E_\Lambda)\bigr)$ 
such that
\begin{equation}\label{eq:Phi-Psi-identity}
\Phi\Psi=I_2,
\qquad
\Theta_{\Lambda,2}(\Psi)=\Psi,
\end{equation}
and
\begin{equation}\label{eq:Psi-HS-global-bound}
\sup_{z\in\D_I}\|\Psi(z)\|_{\mathrm{HS}}
\leq\mathcal C(\delta).
\end{equation}

For \(j\in\Lambda\), let $
P_j:E_\Lambda\longrightarrow\CI^2$
 be the \(j\)-th coordinate projection and set \(\Psi_j:=P_j\Psi\).
Since
\[
P_j\mathcal J_\Lambda=J_2P_j,
\qquad
P_jC_\Lambda=C_2P_j,
\]
the fixed-point identity for \(\Psi\) implies $
\Theta_{2,2}(\Psi_j)=\Psi_j$. 
By Proposition~\ref{prop:fixed-slice-realization} and
\eqref{eq:fixed-point-range}, there is a unique
\(g_j\in H^\infty(\B)\) such that
\begin{equation}\label{eq:gj-realization}
\RI(g_j)=\Psi_j.
\end{equation}

Given
\[
q=x+Ly\in\B,
\qquad
z=x+Iy\in\D_I,
\]
choose the corresponding unit row vector
\(\xi\in\CI^2\) as in
\eqref{eq:spherical-row-formula}. Then $
|g_j(q)|=\|\xi\Psi_j(z)\|$. 
Consequently,
\[
\begin{aligned}
\sum_{j\in\Lambda}|g_j(q)|^2
&=\sum_{j\in\Lambda}\|\xi\Psi_j(z)\|^2 \leq\sum_{j\in\Lambda}\|\Psi_j(z)\|_{\mathrm{HS}}^2 =\|\Psi(z)\|_{\mathrm{HS}}^2
\leq\mathcal C(\delta)^2.
\end{aligned}
\]
This proves \eqref{eq:quat-corona-l2-bound}.

If \(\Lambda=\{1,\ldots,n\}\), then
\[
\sum_{j=1}^n\Phi_j\Psi_j
=\Phi\Psi=I_2.
\]
By multiplicativity and injectivity of \(\RI\),
\[
\sum_{j=1}^n f_j\star g_j=1.
\]

Suppose now that \(\Lambda=\mathbb N\), and let \(Q_N\) be the
orthogonal projection of \(E_\Lambda\) onto its first \(N\) coordinate
blocks. Fix \(0<r<1\) and put
\[
K_r:=\{z\in\D_I:|z|\leq r\}.
\]
The maps \(z\mapsto\Phi(z)^*\) and \(z\mapsto\Psi(z)\) are continuous
in operator norm. Their domains are two-dimensional, and hence
\[
\|T\|_{\op}\leq\|T\|_{\mathrm{HS}}
\leq\sqrt2\,\|T\|_{\op}.
\]
Thus their images of \(K_r\) are compact in the Hilbert--Schmidt norm.
Since \(Q_N\to I_{E_\Lambda}\) strongly,
\(\sup_N\|Q_N\|\leq1\), and uniformly bounded operators converging
strongly to zero converge uniformly on compact subsets, we obtain
\begin{align}
\sup_{z\in K_r}
\|(I_{E_\Lambda}-Q_N)\Phi(z)^*\|_{\mathrm{HS}}
&\longrightarrow0,
\label{eq:countable-Phi-tail}\\
\sup_{z\in K_r}
\|(I_{E_\Lambda}-Q_N)\Psi(z)\|_{\mathrm{HS}}
&\longrightarrow0.
\label{eq:countable-Psi-tail}
\end{align}
For every \(z\in\D_I\), the Cauchy--Schwarz inequality gives
\[
\begin{aligned}
\sum_{j>N}\|\Phi_j(z)\Psi_j(z)\|_{\mathrm{HS}}
&\leq
\left(\sum_{j>N}\|\Phi_j(z)\|_{\mathrm{HS}}^2\right)^{1/2}
\left(\sum_{j>N}\|\Psi_j(z)\|_{\mathrm{HS}}^2\right)^{1/2}\\
&=
\|(I_{E_\Lambda}-Q_N)\Phi(z)^*\|_{\mathrm{HS}}
\,
\|(I_{E_\Lambda}-Q_N)\Psi(z)\|_{\mathrm{HS}}.
\end{aligned}
\]
Hence \(\sum_{j=1}^\infty\Phi_j\Psi_j\) converges absolutely and
uniformly on compact subsets of \(\D_I\), and
\[
\sum_{j=1}^\infty\Phi_j\Psi_j
=\Phi\Psi=I_2.
\]

By \eqref{eq:RI-multiplicative},
\[
\RI(f_j\star g_j)=\Phi_j\Psi_j.
\]
The spherical norm identity therefore gives, for \(q=x+Ly\) and
\(z=x+Iy\),
\[
|(f_j\star g_j)(q)|
\leq\|\Phi_j(z)\Psi_j(z)\|_{\mathrm{HS}}.
\]
It follows that \(\sum_{j=1}^\infty f_j\star g_j\) converges
absolutely and uniformly on compact subsets of \(\B\). Hence its sum
\(h\) is slice regular. On restriction to \(\D_I\), the splitting
components of the partial sums converge locally uniformly to those of
\(h\). Therefore
\[
\RI(h)
=
\sum_{j=1}^\infty\Phi_j\Psi_j
=I_2=\RI(1).
\]
The injectivity of \(\RI\) yields \(h=1\).

Finally, \eqref{eq:quat-corona-l2-bound} implies
\[
\sup_{j\in\Lambda}\|g_j\|_{H^\infty(\B)}
\leq\mathcal C(\delta),
\]
and \eqref{eq:C-def} gives the last estimate in
\eqref{eq:quat-corona-explicit-bound}.
\end{proof}

\begin{remark} 
\label{rem:determinant-comparison}
Suppose that \(\Lambda=\{1,\ldots,n\}\), and write
\(f_{j,I}=F_j+G_jJ\). By the Cauchy--Binet formula,
\(\det(\Phi\Phi^*)\) is the sum of the squared moduli of the
\(2\times2\) minors of
\[
\Phi=(\RI(f_1)\ \cdots\ \RI(f_n)),
\]
and this is precisely the determinant quantity appearing in
\cite[Theorem~1.1]{MR5081054}, written in the present
\(\#\)-notation. By \eqref{eq:operator-corona-equivalence},
\[
\delta^4
\leq
\det\bigl(\Phi(z)\Phi(z)^*\bigr)
\leq1.
\]
Moreover, by \eqref{eq:splitting-norm} and
\eqref{eq:quat-corona-hyp},
\[
\delta^2
\leq
\sum_{j=1}^n
\bigl(|F_j(z)|^2+|G_j(z)|^2\bigr)
\leq1.
\]
Since \(0<\delta<1\), both lower bounds are at least \(\delta^4\).
Thus the hypotheses of \cite[Theorem~1.1]{MR5081054} are recovered
with its parameter replaced by \(\delta^2\). 
\end{remark}

\section{\texorpdfstring{Quaternionic \(H^p\)-corona and Toeplitz corona theorems}
{Quaternionic Hp-corona and Toeplitz corona theorems}}
\label{sec:Hp-Toeplitz-corona}

We now turn to Hardy-space solvability and its operator-theoretic
consequences. The fixed-slice realization developed in the preceding
sections allows us first to prove the \(H^p\)-corona theorem for the
full range \(1\leq p<\infty\). We then use the Hilbert-space structure
of \(H^2(\B)\) to identify quaternionic multipliers with complex
matrix-valued multipliers on a fixed slice. Combining this realization
with Leech's factorization theorem yields a quaternionic factorization
theorem and, as a special case, the Toeplitz-corona theorem.

% ===============================================================
% 6.1. H^p-corona theorem
% ===============================================================

\subsection{\texorpdfstring{The \(H^p\)-corona theorem}
{The Hp-corona theorem}}

\begin{proof}[Proof of Theorem~\ref{thm:quaternionic-Hp-corona}]
Fix \(I,J\in\Sph\) with \(J\perp I\), and let
\[
\Phi
=
(\RI(f_j))_{j\in\Lambda}
\in
H^\infty
\bigl(
\D_I,\mathcal B(E_\Lambda,\CI^2)
\bigr)
\]
be the operator row associated with \((f_j)_{j\in\Lambda}\).
By Lemmas~\ref{lem:operator-row} and~\ref{lem:spherical-row},
\begin{equation}\label{eq:Hp-Phi-bound}
\delta^2 I_2
\leq
\Phi(z)\Phi(z)^*
\leq
I_2,
\qquad z\in\D_I.
\end{equation}

Write $
k_I=H+KJ$
 and set
\begin{equation}\label{eq:Hp-gamma}
\gamma
:=
\RI(k)e_1
=
\binom{H}{-K^\#},
\qquad
e_1=\binom10.
\end{equation}
For \(z=re^{I\theta}\),
\[
\|\gamma(z)\|_{\CI^2}^2
=
|H(z)|^2+|K(\overline z)|^2
\leq
|k(z)|^2+|k(\overline z)|^2.
\]

We shall use the elementary estimate
\begin{equation}\label{eq:reflection-Lp-estimate}
\left(
\frac1{2\pi}
\int_0^{2\pi}
\bigl(
F(\theta)^2+F(-\theta)^2
\bigr)^{p/2}
\,d\theta
\right)^{1/p}
\leq
\rho_p
\left(
\frac1{2\pi}
\int_0^{2\pi}
F(\theta)^p\,d\theta
\right)^{1/p}
\end{equation}
for every nonnegative measurable function \(F\). 
Consequently,
\begin{equation}\label{eq:gamma-Hp-bound}
\|\gamma\|_{H^p(\D_I,\CI^2)}
\leq
\rho_p
\|k\|_{H^p(\B)}.
\end{equation}

Since $
\delta_0^2I_2
\leq
\Phi(z)\Phi(z)^*
\leq
I_2$ 
and \(\delta_0^2\leq\e^{-1}\), we may apply
\cite[Theorem~0.1]{MR2158178} with
\[
E=E_\Lambda,
\qquad
E_*=\C^2,
\qquad
r=\dim E_*=2.
\]
There exists $
v=(v_j)_{j\in\Lambda}
\in H^p(\D_I,E_\Lambda)$ 
such that
\begin{equation}\label{eq:Hp-Phi-v}
\Phi v=\gamma
\end{equation}
and
\begin{equation}\label{eq:Hp-v-estimate}
\|v\|_{H^p(\D_I,E_\Lambda)}
\leq
\mathcal C_0(\delta)
\|\gamma\|_{H^p(\D_I,\CI^2)}.
\end{equation}

Write
\[
v_j=\binom{A_j}{B_j},
\qquad j\in\Lambda,
\]
and let \(h_j\) be the slice regular function determined by
\begin{equation}\label{eq:Hp-hj-definition}
h_{j,I}
=
A_j-B_j^\#J.
\end{equation}
Then
\begin{equation}\label{eq:Hp-RI-hj}
\RI(h_j)
=
\begin{pmatrix}
A_j&-B_j^\#\\
B_j&A_j^\#
\end{pmatrix},
\qquad
\RI(h_j)e_1=v_j.
\end{equation}

We next compare the vector-valued \(H^p\)-norms.
Let
\[
q=re^{L\theta},
\qquad
z=re^{I\theta},
\]
and choose a unit row vector $
\xi=(a,b)\in\CI^2$ 
corresponding to \(L\) as in
\eqref{eq:spherical-row-formula}.
Then
\[
\xi\RI(h_j)(z)
=
\bigl(
aA_j(z)+bB_j(z),
-aB_j^\#(z)+bA_j^\#(z)
\bigr).
\]
Hence, by the spherical row identity,
\begin{align}
\sum_{j\in\Lambda}|h_j(q)|^2
&=
\sum_{j\in\Lambda}
\|\xi\RI(h_j)(z)\|_{\CI^2}^2
\nonumber \leq
\|v(z)\|_{E_\Lambda}^2
+
\|v(\overline z)\|_{E_\Lambda}^2.
\label{eq:Hp-two-column-estimate}
\end{align}
Applying \eqref{eq:reflection-Lp-estimate} with $
F(\theta)
=
\|v(re^{I\theta})\|_{E_\Lambda}$ 
gives
\begin{equation}\label{eq:Hp-h-v-bound}
\|\mathbf h\|_{H^p(\B;\ell^2(\Lambda))}
\leq
\rho_p
\|v\|_{H^p(\D_I,E_\Lambda)}.
\end{equation}
Combining
\eqref{eq:gamma-Hp-bound},
\eqref{eq:Hp-v-estimate},
and
\eqref{eq:Hp-h-v-bound}
gives precisely
\eqref{eq:Hp-corona-estimate}.

If \(\Lambda\) is finite, multiplicativity gives
\[
\begin{aligned}
\RI\left(
\sum_{j\in\Lambda}f_j\star h_j
\right)e_1
&=
\sum_{j\in\Lambda}
\RI(f_j)\RI(h_j)e_1\\
&=
\Phi v
=
\gamma
=
\RI(k)e_1.
\end{aligned}
\]
Since \(h\mapsto\RI(h)e_1\) is injective,
\[
\sum_{j\in\Lambda}f_j\star h_j=k.
\]

Suppose now that \(\Lambda=\mathbb N\), and let \(Q_N\) denote the
orthogonal projection of \(E_\Lambda\) onto its first \(N\)
coordinate blocks. Since \(1\leq p<\infty\),
\[
Q_Nv\longrightarrow v
\qquad
\text{in }H^p(\D_I,E_\Lambda).
\]
Since multiplication by \(\Phi\) is bounded,
\[
\Phi Q_Nv
\longrightarrow
\Phi v=\gamma
\qquad
\text{in }H^p(\D_I,\CI^2).
\]
Put
\[
s_N
:=
\sum_{j=1}^Nf_j\star h_j.
\]
Then
\[
\RI(s_N-k)e_1
=
\Phi Q_Nv-\gamma.
\]
Applying the single-function version of
\eqref{eq:Hp-two-column-estimate}, followed by
\eqref{eq:reflection-Lp-estimate}, gives
\[
\|s_N-k\|_{H^p(\B)}
\leq
\rho_p
\|\Phi Q_Nv-\gamma\|_{H^p(\D_I,\CI^2)}
\longrightarrow0.
\]
Thus the series in
\eqref{eq:Hp-corona-equation} converges to \(k\) in
\(H^p(\B)\).
\end{proof}

% ===============================================================
% 6.2. Hilbert-space specialization
% ===============================================================

\subsection{Leech factorization and the Toeplitz corona theorem}

We now specialize the fixed-slice realization to the Hilbert-space case \(p=2\); recall the coefficient description of \(H^2(\B)\) from Section~\ref{sec:preliminaries}. To formulate the operator-theoretic consequences of the corona
problem, we identify quaternionic \(H^2(\B)\) with a vector-valued
complex Hardy space on a fixed slice.  This realization intertwines
quaternionic \(\star\)-multipliers with ordinary analytic matrix
multipliers and will be used in the Leech factorization theorem below.

\begin{lemma} 
\label{lem:H2-first-column-unitary}
Fix \(I,J\in\Sph\) with \(J\perp I\), and let
\[
\pi_I:\Hq\longrightarrow\CI,
\qquad
\pi_I(\alpha+\beta J)=\alpha.
\]
Regard \(H^2(\B)\) as a complex Hilbert space over \(\CI\) with
inner product
\[
\langle h,u\rangle_I
:=
\pi_I
\bigl(
\langle h,u\rangle_{H^2(\B)}
\bigr).
\]
Then
\[
U:
H^2(\B)
\longrightarrow
H^2(\D_I,\CI^2),
\qquad
Uh:=\RI(h)e_1,
\]
is unitary. Likewise,
\[
U_\Lambda:
H^2(\B;\ell^2(\Lambda))
\longrightarrow
H^2(\D_I,E_\Lambda),
\]
defined by
\[
U_\Lambda(h_j)_{j\in\Lambda}
=
\bigl(
\RI(h_j)e_1
\bigr)_{j\in\Lambda},
\]
is unitary.

Moreover, for every \(f\in H^\infty(\B)\) and
\(h\in H^2(\B)\),
\begin{equation}\label{eq:H2-first-column-intertwining}
U(f\star h)
=
\RI(f)\,Uh.
\end{equation}
\end{lemma}

\begin{proof}
Write
\[
h(q)=\sum_{m\geq0}q^ma_m,
\qquad
u(q)=\sum_{m\geq0}q^mb_m,
\]
where
\[
a_m=\alpha_m+\beta_mJ,
\qquad
b_m=\gamma_m+\delta_mJ,
\]
with
\[
\alpha_m,\beta_m,\gamma_m,\delta_m\in\CI.
\]
The \(m\)-th coefficient of \(Uh\) is
\[
\binom{\alpha_m}{-\overline{\beta_m}}.
\]
Since
\[
\pi_I(\overline{a_m}b_m)
=
\overline{\alpha_m}\gamma_m
+
\beta_m\overline{\delta_m},
\]
the coefficient inner products give
\[
\langle Uh,Uu\rangle_{H^2(\D_I,\CI^2)}
=
\langle h,u\rangle_I.
\]
Thus \(U\) is isometric.

It is also onto. Indeed, if
\[
v=\binom{A}{B}
\in
H^2(\D_I,\CI^2),
\]
then the slice regular function determined by $h_I=A-B^\#J$
 satisfies
\[
Uh=v.
\]
The assertion for \(U_\Lambda\) follows componentwise.
Finally,
\eqref{eq:H2-first-column-intertwining}
follows from the multiplicativity of \(\RI\).
\end{proof}

For \(f\in H^\infty(\B)\), let
\[
M_fh:=f\star h,
\qquad
h\in H^2(\B).
\]
Recall the row multiplier
\(\mathcal M_{\mathbf f}\) defined in
\eqref{eq:row-multiplier-def}.

Under the upper bound
\[
\sum_{j\in\Lambda}|f_j(q)|^2\leq1,
\qquad q\in\B,
\]
Lemma~\ref{lem:operator-row} gives $
\|\Phi\|_\infty\leq1$. 
Lemma~\ref{lem:H2-first-column-unitary} then gives the unitary
intertwining relation
\begin{equation}\label{eq:row-multiplier-intertwining}
U\mathcal M_{\mathbf f}
=
M_\Phi U_\Lambda,
\end{equation}
where $
M_\Phi:
H^2(\D_I,E_\Lambda)
\longrightarrow
H^2(\D_I,\CI^2)$ 
denotes multiplication by \(\Phi\).
In particular,
\[
\|\mathcal M_{\mathbf f}\|\leq1.
\]

We shall use the following multiplier consequence of
Leech's factorization theorem; see
\cite[p.~71, Theorem]{Leech2014}.

\begin{theorem}[Multiplier form of Leech's factorization theorem]
\label{thm:Leech-multiplier}
Let \(\mathcal U,\mathcal V,\mathcal Y\) be complex Hilbert spaces, and let
\[
F\in H^\infty\bigl(\D_I,\mathcal B(\mathcal V,\mathcal Y)\bigr),
\qquad
G\in H^\infty\bigl(\D_I,\mathcal B(\mathcal U,\mathcal Y)\bigr).
\]
Then
\begin{equation}\label{eq:Leech-multiplier-inequality}
M_GM_G^*
\leq
M_FM_F^*
\end{equation}
if and only if there exists $
X\in H^\infty\bigl(\D_I,\mathcal B(\mathcal U,\mathcal V)\bigr)$ 
such that
\begin{equation}\label{eq:Leech-multiplier-factorization}
G=FX,
\qquad
\|X\|_\infty\leq1.
\end{equation}
\end{theorem}

\begin{proof}
The implication from the factorization to the operator inequality is
immediate. Conversely, identify
\[
H^2(\D_I,\mathcal Y)
\oplus
H^2(\D_I,\mathcal V)
\oplus
H^2(\D_I,\mathcal U)
\]
with
\[
H^2\bigl(
\D_I,\mathcal Y\oplus\mathcal V\oplus\mathcal U
\bigr).
\]
Choose a positive constant \(c\) sufficiently large so that the block
operators obtained from \(c^{-1}M_G\) and \(c^{-1}M_F\) are
contractions. Place these operators in the \((1,3)\)- and
\((1,2)\)-blocks, respectively. They commute with multiplication by
\(z\), and the inequality
\eqref{eq:Leech-multiplier-inequality}
gives the corresponding inequality between the two block operators.

Leech's theorem
\cite[p.~71, Theorem]{Leech2014}
therefore gives a contractive factor which also commutes with
multiplication by \(z\). Its \((2,3)\)-block is multiplication by some
\[
X\in
H^\infty\bigl(
\D_I,\mathcal B(\mathcal U,\mathcal V)
\bigr),
\qquad
\|X\|_\infty\leq1.
\]
Consequently,
\[
M_G=M_FM_X,
\]
and hence \(G=FX\).
\end{proof}

\begin{theorem}[Quaternionic Leech factorization theorem]
\label{thm:quaternionic-Leech-factorization}
Let $
f_j\in H^\infty(\B),
j\in\Lambda$ 
satisfy
\begin{equation}\label{eq:Leech-row-upper}
\sum_{j\in\Lambda}
|f_j(q)|^2
\leq1,
\qquad q\in\B,
\end{equation}
and let $
\mathcal M_{\mathbf f}:
H^2(\B;\ell^2(\Lambda))
\longrightarrow
H^2(\B)$  
be the associated row multiplier.
Let $
h\in H^\infty(\B)$ and $ \gamma>0$. 
Then the following statements are equivalent.

\begin{enumerate}
\item[\textup{(i)}]
The inequality
\begin{equation}\label{eq:quaternionic-Leech-domination}
M_hM_h^*
\leq
\gamma^2
\mathcal M_{\mathbf f}
\mathcal M_{\mathbf f}^*
\end{equation}
holds on \(H^2(\B)\).

\item[\textup{(ii)}]
There exist
\[
g_j\in H^\infty(\B),
\qquad
j\in\Lambda,
\]
such that
\begin{equation}\label{eq:quaternionic-Leech-factorization}
h
=
\sum_{j\in\Lambda}
f_j\star g_j,
\end{equation}
and the analytic block function
\begin{equation}\label{eq:quaternionic-Leech-Psi}
\Psi_{\mathbf g}(z)
:=
\begin{pmatrix}
\RI(g_1)(z)\\
\RI(g_2)(z)\\
\vdots
\end{pmatrix}
:
\CI^2
\longrightarrow
E_\Lambda
\end{equation}
belongs to $
H^\infty
\bigl(
\D_I,\mathcal B(\CI^2,E_\Lambda)
\bigr)$ 
and satisfies
\begin{equation}\label{eq:quaternionic-Leech-Psi-bound}
\|\Psi_{\mathbf g}\|_\infty
\leq
\gamma.
\end{equation}
\end{enumerate}

If \(\Lambda=\mathbb N\), the series in
\eqref{eq:quaternionic-Leech-factorization}
converges absolutely and uniformly on compact subsets of \(\B\).

In particular, when \(\Lambda\) is finite,
\eqref{eq:quaternionic-Leech-domination}
gives a Toeplitz-operator characterization, with quantitative
control, of membership of \(h\) in the right ideal generated by
\(f_1,\ldots,f_n\).
\end{theorem}

\begin{proof}
Fix \(I,J\in\Sph\) with \(J\perp I\), and put
\[
\Phi
:=
(\RI(f_j))_{j\in\Lambda},
\qquad
H:=\RI(h).
\]
By Lemma~\ref{lem:H2-first-column-unitary} and
\eqref{eq:row-multiplier-intertwining},
after restricting scalars from \(\Hq\) to \(\CI\),
\(\mathcal M_{\mathbf f}\) is unitarily equivalent to
\[
M_\Phi:
H^2(\D_I,E_\Lambda)
\longrightarrow
H^2(\D_I,\CI^2),
\]
while \(M_h\) is unitarily equivalent to multiplication by
\(H=\RI(h)\):
\[
UM_h=M_HU.
\]
The quaternionic adjoints agree, after restriction of scalars, with
the adjoints for the corresponding complex Hilbert structures.
Consequently,
\eqref{eq:quaternionic-Leech-domination}
is equivalent to
\begin{equation}\label{eq:complex-Leech-domination}
M_HM_H^*
\leq
\gamma^2
M_\Phi M_\Phi^*.
\end{equation}
 
Assume first that \textup{(i)} holds.
Apply Theorem~\ref{thm:Leech-multiplier} with
\[
\mathcal U=\mathcal Y=\CI^2,
\qquad
\mathcal V=E_\Lambda,
\qquad
F=\gamma\Phi,
\qquad
G=H.
\]
By \eqref{eq:complex-Leech-domination}, there exists
\[
X\in
H^\infty
\bigl(
\D_I,\mathcal B(\CI^2,E_\Lambda)
\bigr),
\qquad
\|X\|_\infty\leq1,
\]
such that $
H=\gamma\Phi X$. 
Thus
\[
\Psi_0:=\gamma X
\]
satisfies
\begin{equation}\label{eq:Leech-Psi0}
\Phi\Psi_0=H,
\qquad
\|\Psi_0\|_\infty\leq\gamma.
\end{equation}

The factor \(\Psi_0\) need not satisfy the quaternionic
fixed-point symmetry. Since
\[
\Theta_{2,\Lambda}(\Phi)=\Phi,
\qquad
\Theta_{2,2}(H)=H,
\]
define
\begin{equation}\label{eq:Leech-symmetrization}
\Psi
:=
\frac12
\left(
\Psi_0+
\Theta_{\Lambda,2}(\Psi_0)
\right).
\end{equation}
Applying the product identity
\eqref{eq:Theta-product} to
\(\Phi\Psi_0=H\) gives $
\Phi
\Theta_{\Lambda,2}(\Psi_0)
=
H$.

Therefore
\begin{equation}\label{eq:Leech-fixed-factor}
\Phi\Psi=H,
\qquad
\Theta_{\Lambda,2}(\Psi)=\Psi,
\qquad
\|\Psi\|_\infty\leq\gamma.
\end{equation}

For \(j\in\Lambda\), let $
P_j:E_\Lambda\longrightarrow\CI^2$
 be the \(j\)-th coordinate projection and set $
\Psi_j:=P_j\Psi$. 
As in the proof of Theorem~\ref{thm:quat-corona},
the fixed-point relation for \(\Psi\) implies
\[
\Theta_{2,2}(\Psi_j)=\Psi_j.
\]
Hence, by
Proposition~\ref{prop:fixed-slice-realization}
and \eqref{eq:fixed-point-range}, there is a unique $
g_j\in H^\infty(\B)$ 
such that $
\Psi_j=\RI(g_j)$. 
Thus
\[
\Psi
=
\Psi_{\mathbf g}.
\]

Moreover,
\[
\begin{aligned}
H
&=
\Phi\Psi =
\sum_{j\in\Lambda}
\RI(f_j)\RI(g_j)=
\sum_{j\in\Lambda}
\RI(f_j\star g_j).
\end{aligned}
\]
If \(\Lambda\) is finite, then
\eqref{eq:Leech-fixed-factor} and the multiplicativity of \(\RI\)
give
\[
H
=
\sum_{j\in\Lambda}
\RI(f_j\star g_j).
\]
The injectivity of \(\RI\) therefore yields
\[
h
=
\sum_{j\in\Lambda}
f_j\star g_j.
\]

Suppose now that \(\Lambda=\mathbb N\).
Let \(Q_N\) be the orthogonal projection of \(E_\Lambda\) onto its
first \(N\) coordinate blocks. On every compact subset of \(\D_I\),
the same Hilbert--Schmidt tail argument used in the proof of
Theorem~\ref{thm:quat-corona} gives
\[
\sum_{j=1}^{\infty}
\RI(f_j)\RI(g_j)
\]
absolutely and locally uniformly. By the spherical norm identity,
\[
\sum_{j=1}^{\infty}
f_j\star g_j
\]
therefore converges absolutely and uniformly on compact subsets of
\(\B\). Its sum has matrix realization \(H=\RI(h)\), and hence,
by injectivity of \(\RI\), equals \(h\).
This proves \textup{(ii)}.

Conversely, suppose that \textup{(ii)} holds.
Then
\[
H
=
\Phi\Psi_{\mathbf g}.
\]
Consequently,
\[
M_H
=
M_\Phi M_{\Psi_{\mathbf g}}.
\]
Since
\[
\|M_{\Psi_{\mathbf g}}\|
\leq
\|\Psi_{\mathbf g}\|_\infty
\leq
\gamma,
\]
we obtain
\[
\begin{aligned}
M_HM_H^*
&=
M_\Phi
M_{\Psi_{\mathbf g}}
M_{\Psi_{\mathbf g}}^*
M_\Phi^* \leq
\gamma^2
M_\Phi M_\Phi^*.
\end{aligned}
\]
Transferring this inequality back through the unitary
first-column realization gives
\[
M_hM_h^*
\leq
\gamma^2
\mathcal M_{\mathbf f}
\mathcal M_{\mathbf f}^*,
\]
which proves \textup{(i)}.
\end{proof}

% ===============================================================
% 6.4. Toeplitz corona as a consequence of Leech factorization
% ===============================================================
 
\begin{proof}[Proof of Theorem~\ref{thm:quaternionic-toeplitz-corona}]
Apply Theorem~\ref{thm:quaternionic-Leech-factorization} with
\[
h\equiv1,
\qquad
\gamma=\eta^{-1}.
\]
Since \(M_1M_1^*=I\), condition \textup{(i)} in
Theorem~\ref{thm:quaternionic-Leech-factorization}
is precisely
\eqref{eq:full-toeplitz-lower}.
Its factorization conclusion gives
\eqref{eq:full-toeplitz-bezout}, while the corresponding norm estimate
is exactly
\eqref{eq:full-toeplitz-Psi-bound}.
The convergence assertion for \(\Lambda=\mathbb N\) follows from
Theorem~\ref{thm:quaternionic-Leech-factorization}.
\end{proof}

\subsection*{Conflict of interest}
The authors have no conflict of interest to declare that are relevant to the content of this article. 
\subsection*{Data availability statement}
No data, models, or code were generated or used for the research described in the article.%Data sharing is not applicable to this article as no new data were created or analyzed in this study.
 
%%%%%%%%%%%%%%%%%%%%%%%%%%%%%%%%%%%%%%%%%%%%%%%%%%
\bibliographystyle{amsplain}
\bibliography{references}

\medskip

\noindent
School of Mathematical Sciences, Dalian University of Technology,
Dalian, Liaoning 116024, P. R. China

\noindent
Email address: \texttt{linzhaopeng2606@163.com}

\medskip

\noindent
School of Mathematical Sciences, Dalian University of Technology,
Dalian, Liaoning 116024, P. R. China

\noindent
Email address: \texttt{lyfdlut@dlut.edu.cn}

\medskip

\noindent
School of Mathematical Sciences, Dalian University of Technology,
Dalian, Liaoning 116024, P. R. China

\noindent
Email address: \texttt{zuchao@dlut.edu.cn}
\end{document}